\documentclass[11pt]{amsart}

\usepackage{mystyle}

\title{Subsquares in random Latin squares and rectangles}

\thanks{This project was
supported by the NSF grants \#1745583, \#1851843, \#2244427 and the GaTech
College of Sciences}

\author{Alexander Divoux}
\address{School of Mathematics, Georgia Institute of Technology, Atlanta, GA 30332, USA}
\email{adivoux3@gatech.edu}

\author{Tom Kelly}
\thanks{Kelly's research was supported by the National Science Foundation under Grant No. DMS-2247078.}
\address{School of Mathematics, Georgia Institute of Technology, Atlanta, GA 30332, USA}
\email{tom.kelly@gatech.edu}

\author{Camille Kennedy}
\address{Department of Mathematics, Northwestern University, Evanston, IL 60201, USA}
\email{camillekennedy2025@u.northwestern.edu}

\author{Jasdeep Sidhu}
\address{Department of Mathematics, Stanford University, Stanford, CA 94305, USA}
\email{jasdeep6@stanford.edu}

\begin{document}

\begin{abstract}
A $k \times n$ partial Latin rectangle is \textit{$C$-sparse} if the number of nonempty entries in each row and column is at most $C$ and each symbol is used at most $C$ times. We prove that the probability a uniformly random $k \times n$ Latin rectangle, where $k < (1/2 - \alpha)n$, contains a $\beta n$-sparse partial Latin rectangle with $\ell$ nonempty entries is $(\frac{1 \pm \varepsilon}{n})^\ell$ for sufficiently large $n$ and sufficiently small $\beta$. Using this result, we prove that a uniformly random order-$n$ Latin square asymptotically almost surely has no Latin subsquare of order greater than $c\sqrt{n\log n}$ for an absolute constant $c$.
\end{abstract}

\maketitle

\section{Introduction}\label{sec:intro}

An \textit{order-n Latin square} is an $n \times n$ array of $n$ symbols, such that each row and each column contains each symbol exactly once. For $m < n$, an \textit{order-m subsquare} of an order-$n$ Latin square is an $m \times m$ array induced by a selection of a set of $m$ rows and a set of $m$ columns. Such a subsquare is \textit{Latin} if it is itself a Latin square. The existence and number of order-$m$ Latin subsquares is a structural problem often explored in a probabilistic sense. Let $\mathcal{L}_n$ denote the set of order-$n$ Latin squares with symbol set $[n]$. McKay and Wanless \cite{MW99} proposed the following 1999 conjecture on the number of order-$m$  Latin subsquares in an order-$n$ random Latin square.

\begin{conjecture}[McKay and Wanless \cite{MW99}]\label{conj:McKayWanless}
If $\randLS \in \mathcal{L}_n$ is an order-$n$ Latin square chosen uniformly at random, then as $n \rightarrow \infty$,
\begin{enumerate}[label=(\alph*)]
    \item\label{3x3inLS} the expected number of order-$3$ Latin subsquares of $\randLS$ tends to $1/18$, and 
    \item\label{>3inLS} $\randLS$ asymptotically almost surely contains no Latin subsquare of order greater than $3$. 
\end{enumerate}
\end{conjecture}

McKay and Wanless \cite{MW99} also conjectured that asymptotically almost surely a uniformly random order-$n$ Latin square $\randLS \in \cL_n$ contains close to $n^2 / 4$ order-$2$ Latin subsquares and proved a lower bound of $n^{3/2 + o(1)}$.  In 2008, Cavenagh, Greenhill, and Wanless \cite{CGW08} proved an upper bound of $9n^{5/2}/2$.  In 2018, Kwan and Sudakov \cite{KS18} improved the lower bound to $(1 - o(1))n^2 / 4$, matching the conjecture, and showed that the expected number of order-$2$ Latin subsquares is at most $n^2/2$. Recently, Kwan, Sah, and Sawhney \cite{KS22} settled this conjecture, and Kwan, Sah, Sawhney, and Simkin \cite{KS23} strengthened this result by providing fairly precise estimates for the tail probabilities of the number of order-$2$ Latin subsquares of $\randLS$.

Despite this progress, Conjecture~\ref{conj:McKayWanless} remains wide open.  
As part of an effort to prove Conjecture~\ref{conj:McKayWanless}\ref{>3inLS}, it is natural to seek the smallest $m$, perhaps as a function of $n$, such that asymptotically almost surely a uniformly random order-$n$ Latin square $\randLS \in \cL_n$ contains no Latin subsquare of order greater than $m$.
When they proposed this conjecture, McKay and Wanless proved that asymptotically almost surely $\randLS$ contains no Latin subsquare of order $n/2$ or larger, but no better bound has been proved to date.  
In this paper, we improve this bound, as follows.

\begin{thm}\label{thm:nolss}
    There exists an absolute constant $c > 0$ such that a uniformly random order-$n$ Latin square contains no order-$m$ Latin subsquare, for $m \geq c\sqrt{n\log n}$, asymptotically almost surely.
\end{thm}

To prove Theorem \ref{thm:nolss}, we first prove a result about Latin rectangles. For $k \leq n$, a $k\times n$ \textit{Latin rectangle} is a $k \times n$ array of $n$ symbols, such that each row contains each symbol exactly once and each column contains each symbol at most once.
Let $\setLR$ be the set of $k \times n$ Latin rectangles with symbol set $[n]$. 

Considering Latin rectangles can sometimes be useful for proving results about Latin squares because Latin rectangles are more ``flexible'' structures.  Latin rectangles are also interesting objects of study in their own right.  We generalize Conjecture~\ref{conj:McKayWanless}\ref{3x3inLS} and \ref{conj:McKayWanless}\ref{>3inLS} to all $k \leq n$ as follows.

\begin{conjecture}\label{conj:MWforLRs}
    For all $\eps > 0$, the following holds for sufficiently large $n$ and $k \leq n$.  If $\randLR \in \setLR$ is a $k \times n$ Latin rectangle chosen uniformly at random, then 
    \begin{enumerate}[(a)]
        \item the expected number of order-3 Latin subsquares of $\randLR$ is $\left.(1\pm\eps)\binom{k}{3}\middle/(3n^3)\right.$, and
        \item $\randLR$ contains no Latin subsquare of order greater than 3 with probability at least $1 - \eps$.
    \end{enumerate}
\end{conjecture}

We confirm this conjecture for $k \leq (1/2 - o(1))n$ in the following theorem. 

\begin{thm}\label{thm:MWforLRs}
    For every $\eps, \alpha > 0$, the following holds for sufficiently large $n$ and $k < (1/2 - \alpha)n$.  If $\randLR \in \setLR$ is a $k \times n$ Latin rectangle chosen uniformly at random, then 
    \begin{enumerate}[(a)]
        \item\label{3x3inLR} the expected number of order-3 Latin subsquares of $\randLR$ is $\left.(1\pm\eps)\binom{k}{3}\middle/(3n^3)\right.$, and
        \item\label{>4inLR} $\randLR$ contains a Latin subsquare of order greater than 3 with probability at most $O(n^{-4})$.
    \end{enumerate}
\end{thm}

We prove Theorems~\ref{thm:nolss} and \ref{thm:MWforLRs} from a more general result relating to the degree of independence of entries in random Latin rectangles. 
For $k \leq n$, a $k \times n$ \textit{partial Latin rectangle} is a partially populated $k \times n$ array of $n$ symbols where each symbol appears in each row and column at most once. A \textit{partial order-$n$ Latin square} is a partial $n \times n$ Latin rectangle. 

Let $\randLS$ be a uniformly random order-$n$ Latin square with symbol set $[n]$. By symmetry, the probability $\Prob{\randLR_{i, j} = k}$ that row $i$ and column $j$ of $\randLS$ contains symbol $k$ is $1/n$ for every $i,j,k \in [n]$. However, for $i,j,k,i',j',k' \in [n]$ the events $\randLR_{i,j} = k$ and $\randLR_{i',j'}=k'$ are not independent.  For instance, if $i = i'$ and $k = k'$, then these events are mutually exclusive. Nevertheless, for a sufficiently small set of such events, the events should be close to being mutually independent if they are not mutually exclusive.  That is, for a partial Latin square $P$ with $\ell$ nonempty entries, we should expect $\randLS$ to contain $P$ with probability close to $1 / n^\ell$.  For example, if $i = i'$ and $j \neq j'$ and $k \neq k'$, then the probability that $\randLR_{i, j} = k$ and $\randLR_{i',j'} = k'$ is $1/(n(n - 1)) \approx 1 / n^2$.  However, for $\ell > 2$, this ``probabilistic heuristic'' is challenging to prove, and indeed, if it is true for $\ell = 9$, then it implies Conjecture~\ref{conj:McKayWanless}\ref{3x3inLS}.

This heuristic motivates the numbers in Conjecture \ref{conj:MWforLRs}\ref{3x3inLR}. For a uniformly random $\randLR \in \setLR$, there are $\binom{k}{3}\binom{n}{3}\binom{n}{3}$ ways to choose $3$ rows, columns, and symbols of $\randLR$. There are $12$ ways for the corresponding subsquare to be Latin, and by the probabilistic heuristic there is approximately a $1/n^9$ probability of each of these events occurring. Hence by the linearity of expectation, the expected number of order-$3$ Latin subsquares of $\randLR$ should be near $12\binom{k}{3}\binom{n}{3}\binom{n}{3}/n^9 \approx \binom{k}{3}/(3n^3)$ for sufficiently large $n$. Conjecture \ref{conj:MWforLRs}\ref{>4inLR} is similarly motivated. However, proving approximate mutual independence for a constant $\ell$, or even up to $\ell = o(n)$ entries, would not be sufficient for proving Theorem~\ref{thm:MWforLRs}\ref{>4inLR}.  To that end, we call a partial Latin rectangle \textit{$C$-sparse} when the number of populated entries in each row and column is at most $C$, and each symbol is used at most $C$ times.  We believe the approximate independence for $\randLR$ to contain entries of a $C$-sparse partial Latin rectangle holds for $C = o(n)$.  We formalize this probabilistic heuristic with the following conjecture.

\begin{conjecture}\label{conj:patternversion} 
    For every $\eps > 0$, there exists $\beta > 0$ such that the following holds for sufficiently large $n$ and for all $k \leq n$. If $P$ is a $\beta n$-sparse $k \times n$ partial Latin rectangle with $\ell$ nonempty entries and $\randLR \in \setLR$ is chosen uniformly at random, then 
\[
    \left(\frac{1 - \eps}{n}\right)^\ell \leq \Prob{P \subseteq \randLR} \leq \left(\frac{1 + \eps}{n}\right)^\ell.
\]
\end{conjecture}

The $\beta n$-sparsity condition is essentially necessary in Conjecture~\ref{conj:patternversion}.  For example, if $P$ consists of a single nonempty row or a single nonempty column with $\ell$ nonempty entries, then the probability $\randLR$ contains $P$ is exactly $\ell!/n!$, and if $\ell! / n! \leq ((1 + o(1))/n)^\ell$, then $\ell = o(n)$.  A similar argument can be made for partial Latin rectangles with $\ell$ entries of a fixed symbol.  Our main result in this paper is to confirm Conjecture~\ref{conj:patternversion} for $k \leq (1/2 - o(1))n$, as follows.

\begin{thm}\label{thm:rectanglepatternversion}
For every $\eps, \alpha > 0$, there exists $\beta > 0$ such that the following holds for sufficiently large $n$ and $k < (1/2 - \alpha)n$. If $P$ is a $\beta n$-sparse $k \times n$ partial Latin rectangle with $\ell$ nonempty entries and $\randLR \in \setLR$ is chosen uniformly at random, then
\[
    \left(\frac{1 - \eps}{n}\right)^\ell \leq \Prob{P \subseteq \randLR} \leq \left(\frac{1 + \eps}{n}\right)^\ell.
\]
\end{thm}

From Theorem ~\ref{thm:rectanglepatternversion}, we derive the following bound on the expected number of Latin subsquares in Latin rectangles. For $L \in \mathcal{L}_n$ or $L \in \setLR$ and $r < n$, let $SS_{r}(L)$ denote the number of order-$r$ Latin subsquares of $L$ and let $SS_{r,m}(L)$ denote the number of order-$r$ subsquares that use at most $m$ symbols. Note that $SS_r$ is a random variable that counts the number of Latin subsquares and for every $L \in \mathcal{L}_n$ or $L \in \setLR$, $SS_{r}(L) = SS_{r, r}(L)$ by the definition of a Latin subsquare.
\begin{cor}\label{cor:SSexpectation}
    For every $\eps, \alpha > 0$, there exists $\beta > 0$ such that the following holds for sufficiently large $n$ and $k < (1/2 - \alpha)n$. If $\randLR \in \setLR$ is a uniformly random $k \times n$ Latin rectangle and $m \leq \beta n$, then
    $$\binom{n}{m}^2\binom{k}{m}|\mathcal{L}_{m}|\left(\frac{1 - \eps}{n}\right)^{m^2} \leq \Expect{SS_{m}(\randLR)} \leq \binom{n}{m}^2\binom{k}{m}|\mathcal{L}_{m}|\left(\frac{1 + \eps}{n}\right)^{m^2}.$$
\end{cor}
Theorem~\ref{thm:MWforLRs}\ref{3x3inLR} follows immediately from Corollary~\ref{cor:SSexpectation} since $|\mathcal{L}_{3}| = 12$.  We prove Theorem~\ref{thm:nolss} and Theorem~\ref{thm:MWforLRs}\ref{>4inLR} in Section~\ref{sect:nolss} using Theorem~\ref{thm:rectanglepatternversion}.  We prove Theorem~\ref{thm:rectanglepatternversion} in Section~\ref{sec:proofofgraphthm}.
We also remark that, using the same proof as of Theorem~\ref{thm:MWforLRs} using Theorem~\ref{thm:rectanglepatternversion}, Conjecture~\ref{conj:patternversion}, if true, implies Conjecture~\ref{conj:MWforLRs} (and thus also Conjecture~\ref{conj:McKayWanless}).

Conjecture~\ref{conj:patternversion} (in the case $k = n$) is also closely related to a conjecture of the second author that the uniform distribution on order-$n$ Latin squares is \textit{$((e^2 + o(1))/n)$-spread}, meaning for every partial Latin square $P$ with $\ell$ nonempty entries, the probability that a uniformly random order-$n$ Latin square $\randLS \in \cL_n$ contains $P$ is at most $((e^2 + o(1))/n)^\ell$.  This conjecture, if true, also implies Conjecture~\ref{conj:McKayWanless}\ref{>3inLS} as well as the recent solution to the threshold problem for Latin squares \cite{SSS22, KKKMO22, Ke22, JP22} via the Park--Pham Theorem \cite{PP22}.  The second author also conjectures that the uniform distribution on \textit{$(n,k,t)$-Steiner systems} is $\left.((1 + o(1))\exp\left.(\binom{k}{t}-1\right.)\middle./\binom{n - t}{k - t}\right.)$-spread.  The natural analogue of Conjecture~\ref{conj:patternversion} for Steiner systems would be that every partial $(n,k,t)$-Steiner system of maximum $(t-1)$-degree at most $\beta n$ with $\ell$ blocks is contained in a uniformly random $(n, k, t)$-Steiner system with probability between $((1 - \eps)/\binom{n-t}{k-t})^\ell$ and $((1 + \eps)/\binom{n-t}{k-t})^\ell$.


\section{Preliminaries}

\subsection{Notation}

For $k \in \mathbb N$, we let $[k] = \{1, \dots, k\}$.
For real numbers $a,b,c$ such that $b>0$, we write $a=(1\pm b)c$ to mean that the inequality $(1-b)c\leq a\leq (1+b)c$ holds.
We sometimes state a result with a hierarchy of constants which are chosen from right to left.
If we state that the result holds whenever $a \ll b_1,\dots,b_t$, then this means that there exists a function $f \colon (0,1)^t \to (0,1)$ such that $f(b_1 , \dots , \widetilde{b_i} , \dots , b_t) \leq f(b_1 , \dots , b_i , \dots , b_t)$ for $0 < \widetilde{b_i} \leq b_i < 1$ for all $i \in [t]$ and the result holds for all real numbers $0 < a , b_1 , \dots , b_t < 1$ with $a \leq f(b_1 , \dots , b_t)$. If a reciprocal $1/m$ appears in such a hierarchy, we implicitly assume that $m$ is a positive integer.
Hierarchies with more constants are defined similarly and should be read from the right to the left.

\subsection{Graph theory}\label{sect:graph-theory}

A \textit{matching} in a graph $G$ is a set of pairwise disjoint edges of $G$, and a \textit{perfect matching} of a graph $G$ is a matching in $G$ containing all of its vertices. We will sometimes treat a matching $M$ as if it is a graph in which every vertex has degree one, using $V(M)$ to refer to its vertex set.

For $k \leq n$, let $\cN_{k, n}$ be the set of $(N_1, \dots, N_k)$, where $N_1, \dots, N_k$ are pairwise edge-disjoint perfect matchings of $K_{n,n}$. 
There exists a one-to-one correspondence between $\cN_{k,n}$ and $\setLR$: given an $L \in \setLR$, we construct a tuple $(N_1, N_2, \dots, N_k) \in \cN_{k,n}$ as follows. Label the vertices in one partite set $A$ of $K_{n, n}$ as the symbols $s_1, \dots, s_n$ of $L$, and label the vertices in the other partite set $B$ of $K_{n, n}$ as the columns $c_1, \dots, c_n$ of $L$. For $i \in [k]$, consider the $i$-th row of $L$. In this row, let symbol $s_\ell$ appear in the $j$-th column. For each row $i$, create a perfect matching $N_i$ where $c_j$ is matched with $s_\ell$ if $s_\ell$ appears in the $j$-th column of the $i$-th row of $L$. For all $i \in [k]$, $N_i$ is indeed a perfect matching because every symbol appears exactly once in row $r_i$ and every column intersects $r_i$ exactly once, and hence every $s \in A$ and $c \in B$ is also in $N_i$ and has degree one. Furthermore, any two distinct perfect matchings $N_i$ and $N_j$ are edge-disjoint by the Latin property of $L$. The other direction of the correspondence follows from identical (but reversed) reasoning.

If $M_1, \dots, M_k$ are pairwise edge-disjoint matchings (not necessarily perfect) of $K_{n,n}$, then we say $(M_1, \dots, M_k)$ is \textit{$C$-sparse} if $\bigcup_{i=1}^kM_i$ has maximum degree at most $C$ and $|M_i| \leq C$ for all $i \in [k]$.  Note that in this case, $(M_1, \dots, M_k)$ 
corresponds to a partial $k \times n$ Latin rectangle that is $C$-sparse. 

\subsection{Robust expansion}

Let $\nu, \tau > 0$ and let $D$ be a directed graph on $n$ vertices. The $\nu$\textit{-robust outneighborhood} $RN_{\nu, D}^{+}(S)$ of a set $S \subseteq V(D)$ in $D$ consists of all the vertices of $D$ which have at least $\nu n$ in-neighbors in $S$. We say that $D$ is a \textit{robust} $(\nu, \tau)$\textit{-outexpander} if $RN_{\nu, D}^{+}(S) \geq |S| + \nu n$ for each $S \subseteq V(D)$ satisfying $\tau n \leq |S| \leq (1 - \tau)n$. A digraph $D$ on $n$ vertices is $(\delta, f)$\textit{-almost regular} if all vertices in $D$ have in-degree and out-degree $(1 \pm f)\delta n$. 

We will need the following lemma for the proof of Theorem~\ref{thm:rectanglepatternversion}.  Using the fact that random walks mix rapidly in nearly regular robust expanders, it provides a precise estimation for the number of long paths between two fixed vertices or cycles containing some fixed vertex.

\begin{restatable}{lemma}{granetjoos}\label{lem:granetjoos2.5}
    Let $0 < {1}/{n} \ll \eps \ll 1/\ell \ll \gamma \ll \nu \leq \tau \ll \delta \leq 1$ and $1 / n \leq f \leq \eps$. If $G$ is a $(\delta, f)$-almost regular robust $(\nu, \tau)$-outexpander on $n$ vertices, then for any distinct $u, v \in V(G)$, there exist $(1 \pm \gamma)\delta^\ell n^{\ell-1}$
    $(u, v)$-paths $P$ of length $\ell$ and for every $u \in V(G)$, there exist $(1 \pm \gamma)\delta^\ell n^{\ell - 1}$ length-$\ell$ cycles containing $u$.
\end{restatable}

Lemma~\ref{lem:granetjoos2.5} is very similar to a recent result of Granet and Joos \cite[Lemma 2.5]{GJ23} and has a similar proof. We thus include a proof of Lemma~\ref{lem:granetjoos2.5} in Appendix~\ref{sec:appendix}.


\section{Proof of Theorem~\ref{thm:rectanglepatternversion}}\label{sec:proofofgraphthm}

In this section, we prove Theorem~\ref{thm:rectanglepatternversion}.  First, we deduce Theorem~\ref{thm:rectanglepatternversion} from the following lemma.

\begin{lemma}\label{lem:prob1elem}
    Let $0 < 1 / n \ll \beta \ll \eps, \alpha < 1$, and let $k < (1/2 - \alpha)n$.  Let $P$ and $P'$ be partial $k \times n$ Latin rectangles such that $P \subseteq P'$ and $P'$ has one more nonempty entry than $P$.  If $P'$ is $\beta n$-sparse and $\randLR \in \setLR$ is chosen uniformly at random, then
    \begin{equation*}
        \ProbCond{P' \subset \randLR}{P \subset \randLR} = \frac{1 \pm \eps}{n}.
    \end{equation*}
\end{lemma}

\begin{proof}[Proof of Theorem~\ref{thm:rectanglepatternversion} assuming Lemma~\ref{lem:prob1elem}]
    Let $P$ be a $\beta n$-sparse $k \times n$ partial Latin rectangle with $\ell$ nonempty entries.  Let $P_1, \dots, P_\ell$ be $k \times n$ partial Latin rectangles satisfying $P_1 \subset \cdots \subset P_\ell = P$, where $P_i$ has $i$ nonempty entries for every $i \in [\ell]$.  By the chain rule of probability,
    \begin{equation*}
        \Prob{P \subseteq \randLR} = \Prob{\bigcap_{i=1}^\ell P_i \subseteq \randLR} = \prod_{i = 1}^{\ell} \ProbCond{P_i \subseteq \randLR}{P_{i-1} \subseteq \randLR}.
    \end{equation*}
    Since $P$ is $\beta n$-sparse, $P_i$ is $\beta n$-sparse for all $i \in [\ell]$, so by Lemma~\ref{lem:prob1elem}, 
    \begin{equation*}
        \ProbCond{P_i \subseteq \randLR}{P_{i-1} \subseteq \randLR} = \frac{1 \pm \eps}{n}
    \end{equation*}
    for every $i \in [\ell]$.  Combining the two equations above, we have $\Prob{P \subseteq \randLR} = ((1\pm\eps)/n)^\ell$, as desired.
\end{proof}

The rest of the section will be dedicated to proving Lemma~\ref{lem:prob1elem}. 
It is easy to see that Lemma~\ref{lem:prob1elem} holds when $P'$ has one nonempty entry, and in fact, in this case $\ProbCond{P' \subset \randLR}{P \subset \randLR} = \Prob{P' \subset \randLR} = 1 / n$.  That is, given a row $r$, column $c$, and symbol $s$, there is precisely a $1 / n$ chance that $s$ is used in row $r$ and column $c$ in a uniformly random Latin rectangle.  Lemma~\ref{lem:prob1elem} states that this fact also holds when $\randLR$ is chosen from the conditional distribution, up to a minor error term, when we condition on $\randLR$ containing some $\beta n$-sparse partial Latin rectangle.  

Instead of working with Latin rectangles in the proof of Lemma~\ref{lem:prob1elem}, we will work with graphs. Due to the correspondence between Latin rectangles and edge-disjoint matchings in $K_{n,n}$ discussed in Section~\ref{sect:graph-theory}, the following lemma is equivalent to Lemma~\ref{lem:prob1elem}.

\begin{lemma}\label{lem:prob1elem-graphtheory}
    Let $0 < 1 / n \ll \beta \ll \eps, \alpha < 1$, and let $k < (1/2 - \alpha)n$.  Let $M_1, \dots, M_k$ be pairwise edge-disjoint matchings of $K_{n,n}$, and let $M'_1, \dots, M'_k$ be pairwise edge-disjoint matchings of $K_{n,n}$ such that $M_i \subseteq M'_i$ for all $i \in [k]$ and $\bigcup_{i=1}^k M'_i$ has one more edge than $\bigcup_{i=1}^k M_i$.  If $(M'_1, \dots, M'_k)$ is $\beta n$-sparse and $(\randPM_1, \dots, \randPM_k) \in \cN_{k,n}$ is chosen uniformly at random, then
     \begin{equation*}
        \ProbCond{M'_i \subseteq \randPM_i~\forall i\in[k]}{M_i \subset \randPM_i~\forall i\in[k]} = \frac{1 \pm \eps}{n}.
     \end{equation*}
\end{lemma}

The main ingredient in the proof of Lemma~\ref{lem:prob1elem-graphtheory} is the following lemma, which uses Lemma~\ref{lem:granetjoos2.5} and a ``switching'' argument.
Roughly, this lemma states that in a uniformly random $k \times n$ Latin rectangle for $k < (1/2 - \alpha)n$, given a row $r$, column $c$ and symbol $s$, the probability the symbol $s$ does not appear in column $c$ is close to $n - k$ times the probability $s$ is used in row $r$ and column $c$.  Moreover, this same result holds for the conditional distribution if we condition on the random Latin rectangle containing some $\beta n$-sparse partial Latin rectangle. 

\begin{lemma}\label{lem:switching}
    Let $0 < 1/n \ll \beta \ll \eps, \alpha < 1$, and let $k < (1/2 - \alpha)n$.  
    Let $M_1, \dots, M_k$ be pairwise edge-disjoint matchings of $K_{n,n}$, and let $\cN \subseteq \cN_{k,n}$ be the set of $(N_1, \dots, N_k)\in\cN_{k,n}$ satisfying $M_i \subseteq N_i$ for all $i \in [k]$.
    Let $e$ be an edge of $K_{n,n}$ not in $\bigcup_{i=1}^k M_i$, and for each $j \in [k]$, let $A_j \subseteq \cN$ be the set of $(N_1, \dots, N_k) \in \cN$ such that $e \in N_j$.
    For every $j \in [k]$, if $|M_j| \leq \beta n$ and $e$ is disjoint from the edges of $M_j$, then
    \begin{equation*}
        (n - k)|A_j| = (1 \pm \eps)|B|,
    \end{equation*}
    where $B = \cN \setminus \bigcup_{i=1}^k A_i$.  
\end{lemma}
\begin{proof}     
    We assume $k = \lfloor (1/2 - \alpha)n\rfloor$.  Let $j \in [k]$ where $|M_j| \leq \beta n$ and $e$ is disjoint from the edges of $M_j$.
    We use a switching argument to bound the relative sizes of $A_j$ and $B$. Let $\ell \in \mathbb{N}$ and $\gamma, \nu,\tau \in (0,1)$ such that $\beta \ll 1/\ell \ll \gamma \ll \nu \ll \tau \ll  \eps, \alpha$. For $(N_1, \dots, N_k) \in A_j$, we define a $\textit{switch}$ of $(N_1, \dots, N_k)$ as a length-$(2\ell)$ cycle containing $e$ consisting of alternating edges in $N_j$ and $K_{n, n} \setminus \bigcup_{i = 1}^k N_i$. We construct an auxiliary bipartite graph $H$ with bipartition $\{A_j, B\}$ where for $(N_1, \dots, N_k) \in A_j$ and $(N'_1, \dots, N'_k) \in B$, we have $(N_1, \dots, N_k)$ adjacent to $(N'_1, \dots, N'_k)$ in $H$ if and only if $N_i = N'_i$ for all $i \in [k]\setminus\{j\}$ and $N_j \triangle N'_j$ is a switch of $(N_1, \dots, N_k)$.
    
    Let $S$ and $C$ be the parts of the bipartition of $K_{n,n}$ (corresponding to symbols and columns, respectively). For each $N = (N_1, \dots, N_k) \in \mathcal{N}$, we create a digraph $D_{N}$ with $V(D_{N}) = S\setminus V(M_j)$ where for $u,v\in V(D_{N})$, we have an arc from $u$ to $v$ in $D_{N}$ if and only if there exists $x \in C$ such that $ux \in K_{n, n}\setminus \bigcup_{i=1}^k N_i$ and $xv \in N_j$. Note that for all $N \in \cN$, the digraph $D_n$ has $(1 \pm \beta)n$ vertices, minimum semidegree at least $n - k - |M_j| \geq (1/2 + \alpha - \beta)n$ and maximum semidegree at most $n - k \leq (1/2 + \alpha + \beta)n$.  Hence, $D_{N}$ is  $(1/2 + \alpha, 2\beta)$-almost regular, and by a result of K\"uhn and Osthus \cite[Lemma 13.2]{KO13} (with $\alpha - \beta$ and $n - |M_j|$ playing the roles of $\eps$ and $n$, respectively), $D_{N}$ is a robust $(\nu, \tau)$-outexpander.

    Let the endpoints of $e$ be $u \in S$ and $v \in C$. For every $N \in A_j$, note that $d_{H}(N)$ is the number of length-$\ell$ cycles in $D_{N}$ containing $u$. By Lemma~\ref{lem:granetjoos2.5} with $1/2 + \alpha$ and $2\beta$ playing the roles of $\delta$ and $f$, respectively, there are $(1 \pm \gamma)(1/2 + \alpha)^\ell|V(D_N)|^{\ell - 1} = (1\pm 2\gamma)(1/2 + \alpha)^\ell n^{\ell - 1}$ such cycles. For $N = (N_1, \dots, N_j) \in B$, similarly $d_H(N)$ is the number of length-$(\ell - 1)$ paths in $D$ from $x$ to $u$, where $xv \in E(N_j)$. By Lemma~\ref{lem:granetjoos2.5} with the same parameters except $\ell - 1$ playing the role of $\ell$, there are $(1 \pm 2\gamma)(1/2 + \alpha)^{\ell - 1}n^{\ell - 2}$ such paths.

    We now use a double counting argument to bound the relative sizes of $A_j$ and $B$ by counting edges in $H$. Since $H$ is bipartite,
    \[
        \sum_{N \in A_j}{d_{H}(N)} = |E(H)| = \sum_{N \in B}d_H(N).
    \]
    Combining the equality above with our lower bound on $d_H(N)$ for $N \in A_j$ and our upper bound on $d_H(N)$ for $N \in B$,
    \[
        |A_j|(1 - 2\gamma)(1/2 + \alpha)^{\ell}n^{\ell - 1} \leq |B|(1 + 2\gamma)(1/2 + \alpha)^{\ell - 1}n^{\ell - 2}.
    \]
    Combining the equality above with our upper bound on $d_H(N)$ for $N \in A_j$ and our lower bound on $d_H(N)$ for $N \in B$,
    \[
        |A_j|(1 + 2\gamma)(1/2 + \alpha)^{\ell}n^{\ell - 1} \geq |B|(1 - 2\gamma)(1/2 + \alpha)^{\ell - 1}n^{\ell - 2}.
    \]
    Combining the inequalities above, we have
    \[
        \frac{1 - 2\gamma}{1 + 2\gamma}|B| \leq (1/2 + \alpha)n|A_j| \leq \frac{1 + 2\gamma}{1 - 2\gamma}|B|,
    \]
    and since $\gamma \ll \eps$ and $k = \lfloor (1/2 - \alpha)n\rfloor$, we have
    \[
        (1-\eps)|B| \leq (n-k)|A_j| \leq (1+\eps)|B|,
    \]
    as desired.
\end{proof} 

Now we can prove Lemma~\ref{lem:prob1elem-graphtheory}, which we note also yields a proof of Lemma~\ref{lem:prob1elem}.  

\begin{proof}[Proof of Lemma~\ref{lem:prob1elem-graphtheory}]
    Let $e$ be the edge of $\bigcup_{i=1}^k M'_i$ that is not in $\bigcup_{i=1}^k M_i$, and let $\ell \in [k]$ be the unique integer for which $e \in M'_\ell$.  
    Let $\cN \subseteq \cN_{k,n}$ be the set of $(N_1, \dots, N_k)\in\cN_{k,n}$ satisfying $M_i \subseteq N_i$ for all $i \in [k]$.
    By a result of Bowditch and Dukes \cite[Corollary 1.4]{BD19}, since the partial Latin rectangle represented by $(M_1, \dots, M_k)$ is $\beta n$-sparse where $\beta \ll 0.04$, it can be completed into a Latin rectangle.  Hence, there is at least one $(N_1, \dots, N_k) \in \cN_{k,n}$ satisfying $M_i \subseteq N_i$ for all $i \in [k]$, so $|\mathcal{N}| > 0$. 
    For each $j \in [k]$, let $A_j \subseteq \cN$ be the set of $(N_1, \dots, N_k) \in \cN$ such that $e \in N_j$, and let $B = \cN \setminus \bigcup_{i=1}^k A_i$.  Note that
    \begin{equation}\label{eqn:probability-main}
        \ProbCond{M'_i \subseteq \randPM_i~\forall i\in[k]}{M_i \subseteq \randPM_i~\forall i\in[k]} = \frac{|A_\ell|}{|\cN|}
    \end{equation}
    and
    \begin{equation}\label{eqn:size-of-N}
        |\cN| = |B| + \sum_{i=1}^k |A_i|.
    \end{equation}

    Let $\mathcal{I}$ be the set of $i \in [k]$ such that $e$ is not disjoint from the edges of $M_i$. 
    Since $\bigcup_{i=1}^{k}{M_i}$ is $\beta n$-sparse, each endpoint of $e$ is saturated by at most $\beta n$ matchings of $M_1, \dots, M_k$, so
    \begin{equation}\label{eqn:I-is-small}
        |\mathcal{I}| \leq 2\beta n.
    \end{equation}
    Moreover,
    \begin{equation}\label{eqn:no-extensions-for-I}
        |A_i| = 0 \text{ for every } i \in \mathcal I,
    \end{equation}
    because no matching containing $M_i$ for $i \in \mathcal I$ contains $e$.

    Let $\eps' \in (0, 1)$ such that $\beta \ll \eps' \ll \eps, \alpha$. By Lemma~\ref{lem:switching} with the same parameters except $\eps'$ playing the role of $\eps$, 
    \begin{equation}\label{eqn:applying-switching}
        |A_i| = \frac{1\pm\eps'}{n-k}|B| \text{ for every } i \in [k]\setminus \mathcal I,
    \end{equation}
    since $e$ is disjoint from $M_i$ for $i \in [k]\setminus \mathcal I$.
    
    Partitioning the summation $\sum_{i = 1}^{k}{|A_i|}$ around $\mathcal{I}$, by \eqref{eqn:size-of-N} and \eqref{eqn:no-extensions-for-I}, we have
    \[
        |\mathcal{N}| = |B| + \sum_{i \in \mathcal{I}}{|A_i|} + \sum_{j \in [k]\setminus \mathcal{I}}{|A_j|} = |B| + \sum_{i \in [k]\setminus \mathcal{I}}{|A_i|}.
    \]
    This equation, together with \eqref{eqn:I-is-small} and \eqref{eqn:applying-switching}, implies that 
    \[
        |B| + (k-2\beta n)\left(\frac{1-\eps'}{n-k}\right)|B| \leq |\mathcal{N}| \leq |B| + k\left(\frac{1+\eps'}{n-k}\right)|B|.
    \]
    Since $e \in M'_\ell$ and $M_\ell = M'_\ell\setminus\{e\}$, we have $\ell \in [k]\setminus \mathcal I$, so the previous inequality together with \eqref{eqn:applying-switching}, using also that $k \leq n / 2$, yields the upper bound
    \[
        \frac{|A_\ell|}{|\cN|} \leq \frac{(1 + \eps')|B|/(n - k)}{|B|+(k-2\beta n)\left(\frac{(1-\eps')|B|}{n-k}\right)} = \frac{1 + \eps'}{n-k+(k-2\beta n)(1-\eps')} \leq \frac{1 + \eps'}{n(1 - 2\beta n - \eps'/2)}
    \]
    and the lower bound
    \[
        \frac{|A_\ell|}{|\cN|} \geq \frac{(1-\eps')|B|/(n-k)}{|B|+k\left(\frac{(1+\eps')|B|}{n-k}\right)} = \frac{1-\eps'}{n-k+k(1+\eps')} \geq \frac{1 + \eps'}{n(1 + \eps'/2)}.
    \]
    As $\beta \ll \eps' \ll \eps, \alpha$, by the previous two inequalities and \eqref{eqn:probability-main}, we have
    \[
        \ProbCond{M'_i \subseteq \randPM_i~\forall i\in[k]}{M_i \subseteq \randPM_i~\forall i\in[k]} = \frac{1\pm\eps}{n}
    \]
    as desired.
\end{proof}

We conclude this section with a brief discussion of the condition $k \leq (1/2 - \alpha)n$ in Theorem~\ref{thm:rectanglepatternversion}.  In Lemma~\ref{lem:switching}, we constructed an auxiliary digraph from a $k \times n$ Latin rectangle with $n$ vertices and minimum semidegree close to $n - k$.  For $k \leq (1/2 - \alpha)n$, this digraph is a robust outexpander; however, for $k > n/2$, this digraph may not even be connected, and our approach breaks down.  On the other hand, if we assume $k = o(n)$, then a simpler argument works for Lemma~\ref{lem:switching} without considering robust expansion.  In this case, we could use switches on cycles of length six.


\section{Proof of Theorems \ref{thm:nolss} and \ref{thm:MWforLRs}}\label{sect:nolss}

In this section, we prove Theorems \ref{thm:nolss} and \ref{thm:MWforLRs}.
Recall Theorem~\ref{thm:MWforLRs}\ref{3x3inLR} follows from Corollary~\ref{cor:SSexpectation}, so we only need to prove Theorem~\ref{thm:MWforLRs}\ref{>4inLR}.  First we use Theorem~\ref{thm:rectanglepatternversion} to bound the expected number of order-$r$ subsquares containing at most $m$ symbols in a uniformly random Latin rectangle, as follows.

\begin{lemma}\label{lem:expectation-of-lss-in-rectangle}
    Let $1/n \ll \beta  \ll \eps, \alpha < 1$, and let $k < (1/2 - \alpha)n$. 
    If $\randLR \in \setLR$ is a uniformly random $k \times n$ Latin rectangle and $m, r \in [n]$ where $r \leq \beta n$, then
    \begin{equation*}
         \Expect{SS_{r, m}(\randLR)} \leq \left((1 + \eps)\frac{m}{n}\right)^{r^2}\binom{n}{r}\binom{k}{r}\binom{n}{m}.
    \end{equation*} 
\end{lemma}
\begin{proof}

Since $r \leq \beta n$, a partial $k \times n$ Latin rectangle with at most $r$ nonempty rows and columns is $\beta n$-sparse. Therefore, by linearity of expectation and Theorem~\ref{thm:rectanglepatternversion}, since there are at most $\binom{n}{r}\binom{k}{r}\binom{n}{m}m^{r^2}$ partial $k \times n$ Latin rectangles with $r$ nonempty rows and columns and at most $m$ symbols in a $k \times n$ Latin rectangle, 
    \begin{equation*}
         \Expect{SS_{r, m}(\randLR)} \leq \left((1 + \eps)\frac{m}{n}\right)^{r^2}\binom{n}{r}\binom{k}{r}\binom{n}{m},
    \end{equation*} 
as desired.
\end{proof}

Using Lemma~\ref{lem:expectation-of-lss-in-rectangle}, we can bound the expected number of order-$m$ Latin subsquares in a uniformly random Latin rectangle.  
First, we show that the expected number of large Latin subsquares is extremely small.  For $m > \beta n$, these subsquares are not $\beta n$-sparse, so we cannot directly apply Theorem~\ref{thm:rectanglepatternversion}.  However, they contain $\beta n$-sparse subsquares that are unlikely to appear.  Here, we use that the maximum order of a proper Latin subsquare in a $k\times n$ Latin rectangle is at most $n / 2$, so in particular, every proper Latin subsquare uses at most $n / 2$ symbols.

\begin{lemma}\label{lem:prob-of-large-lss}
    Let $1/n \ll \alpha < 1$, and let $k < (1/2 - \alpha)n$. 
    If $\randLR \in \setLR$ is a uniformly random $k \times n$ Latin rectangle and $m \geq n^{3/4}$, then
    \begin{equation*}
        \Expect{SS_{m}(\randLR)} \leq \exp(n^{-3/2}/10).
    \end{equation*}
\end{lemma}
\begin{proof}
    Let $r = \lceil n^{3/4} \rceil$.  Since $m \geq n^{3/4}$,
    \begin{equation*}
        \Expect{SS_m(\randLR)} \leq \Expect{SS_{r, m}(\randLR)}.
    \end{equation*}
    By Lemma~\ref{lem:expectation-of-lss-in-rectangle}, since $\binom{n}{r},\binom{k}{r},\binom{n}{m} \leq 2^n$, 
    \begin{equation*}
          \Expect{SS_{r, m}(\randLR)} \leq \left((1 + \eps)\frac{m}{n}\right)^{r^2}\binom{n}{r}\binom{k}{r}\binom{n}{m} \leq \left((1 + \eps)\frac{m}{n}\right)^{n^{3/2}}8^n.
    \end{equation*}
    Since the maximum order of a proper Latin subsquare of $\randLS$ is $n/2$, we may assume $m \leq n/2$.  Therefore,
    \begin{equation*}
    \left((1 + \eps)\frac{m}{n}\right)^{n^{3/2}}8^n \leq \exp\left(-n^{3/2}\log\left(\frac{2}{1 + \eps}\right) + n \log 8\right).
    \end{equation*}
    Since $1 + \eps < 3/2$, we have that $\log(2 / (1 + \eps) ) \geq \log(4/3) \geq 1/9$. Therefore, combining the inequalities above, we have
    \begin{equation*}
        \Expect{SS_m(\randLR)} \leq \exp\left(-n^{3/2} / 9 + n \log 8\right) \leq \exp\left(-n^{3/2} / 10\right),
    \end{equation*}
    as desired. 
\end{proof}

With an argument similar to that of the previous lemma, we can use Lemma~\ref{lem:expectation-of-lss-in-rectangle} to bound the expected number of order-$m$ Latin subsquares in a uniformly random Latin rectangle for $m < n^{3/4}$, as follows.

\begin{lemma}\label{lem:prob-of-small-lss}
    Let $1/n \ll \alpha < 1$, and let $k < (1/2 - \alpha)n$. 
    If $\randLR \in \setLR$ is a uniformly random $k \times n$ Latin rectangle and $m \leq n^{3/4}$, then
    \begin{equation*}
        \Expect{SS_{m}(\randLR)} \leq 3^{m^2}\left(\frac{m}{n}\right)^{m^2 - 3m}.
    \end{equation*}
    Moreover, if $m' \geq m \geq 2$ and $m' \leq n^{3/4}$, then
    \begin{equation*}
        \Expect{SS_{m'}(\randLR)} \leq 3^{m^2}\left(\frac{m}{n}\right)^{m^2 - 3m}.
    \end{equation*}
\end{lemma}
\begin{proof}
Let $1 / n \ll \eps \ll \alpha$.  
By Lemma~\ref{lem:expectation-of-lss-in-rectangle} with $m$ playing the role of $r$ and $m$, since $e(1 + \eps) \leq 3$, 
    $$\Expect{SS_m(\randLR)} \leq \left((1 + \eps)\frac{m}{n}\right)^{m^2}\binom{n}{m}^3 \leq \left((1 + \eps)\frac{m}{n}\right)^{m^2}\left(\frac{en}{m}\right)^{3m} \leq 3^{m^2} \left(\frac{m}{n}\right)^{m^2 - 3m},$$
    as desired.
    
Now we show the function $m \mapsto 3^{m^2}(m / n)^{m^2 - 3m}$ is decreasing in $m$ for $2 \leq m \leq n^{3/4}$.  The derivative of this function with respect to $m$ is
\begin{equation*}
    3^{m^2}\left(\frac{m}{n}\right)^{m^2 - 3m}\left(m\left(2\log 3 + 1\right) + (2m - 3)\log\left(\frac{m}{n}\right) - 3\right).
\end{equation*}
For sufficiently large $n$, since $2 \leq m \leq n^{3/4}$, 
\begin{equation*}
    m(2\log 3 + 1) + (2m - 3)\log\left(\frac{m}{n}\right) - 3 \leq m(2\log 3 + 1) - (m/2)\log (n^{1/4}) < 0, 
\end{equation*}
so the function is decreasing, as claimed.  Hence, if $m' \geq m \geq 2$ and $m' \leq n^{3/4}$, then
\begin{equation*}
    \Expect{SS_{m'}(\randLR)} \leq 3^{m^2}\left(\frac{m}{n}\right)^{m^2 - 3m},
\end{equation*}
as desired.
\end{proof}

By combining Lemmas \ref{lem:prob-of-large-lss} and \ref{lem:prob-of-small-lss}, we can prove Theorem~\ref{thm:MWforLRs}\ref{>4inLR}.

\begin{proof}[Proof of Theorem~\ref{thm:MWforLRs}\ref{>4inLR}]
    Let $\randLR \in \setLR$ be a uniformly random $k \times n$ Latin rectangle, and let $1 / C \ll 1$.
    By Markov's Inequality, $\randLR$ contains a Latin subsquare of order greater than $3$ with probability at most
    \begin{equation*}
        \sum_{m=4}^{\lfloor n / 2\rfloor}\Expect{SS_m(\randLR)}.
    \end{equation*}
   By Lemma~\ref{lem:prob-of-large-lss},
   \begin{equation*}
       \sum_{m=\lceil n^{3/4}\rceil}^{\lfloor n / 2\rfloor}\Expect{SS_m(\randLR)} \leq n\exp(n^{-3/2}/10) \leq n^{-4}.
   \end{equation*}
   By Lemma~\ref{lem:prob-of-small-lss} with $5$ playing the role of $m$,
   \begin{equation*}
        \sum_{m=5}^{\lfloor n^{3/4}\rfloor}\Expect{SS_m(\randLR)} \leq n3^{5^2}\left(\frac{5}{n}\right)^{5^2 - 15} \leq n^{-4},
   \end{equation*}
   and again by Lemma~\ref{lem:prob-of-small-lss} with $4$ playing the role of $m$,
   \begin{equation*}
       \Expect{SS_4(\randLR)} \leq 3^{4^2}\left(\frac{4}{n}\right)^{4^2 - 12} \leq Cn^{-4}.
   \end{equation*}
   The result follows.
\end{proof}


The remainder of the section is devoted to the proof of Theorem~\ref{thm:nolss}.
 For the following lemma we will introduce the idea of an induced Latin rectangle. For $L \in \mathcal{L}_n$ and $R \subset [n]$, let $L|_{R}$ denote the $|R| \times n$ Latin rectangle induced by the rows indexed by $R.$
\begin{lemma}\label{lem:expectedvalue}
Let $n \in \mathbb{N}$ and $m, k \in [n]$. If $\randLS \in \mathcal{L}_n$ is a uniformly random order-$n$ Latin square, then 

$$\Expect{SS_{m}(\randLS)}\binom{n - m}{k - m} = \Expect{SS_{m}(\randLS|_{[k]})}\binom{n}{k}.$$
\end{lemma}

\begin{proof} 
First, note that
\begin{equation}
    \begin{split}
        \Expect{SS_{m}(\randLS|_{[k]})} &= \frac{1}{|\mathcal{L}_n|}\sum_{L \in \mathcal{L}_n}\sum_{M \in {\binom{[k]}{m}}}SS_m(L|_{M}) \\
        &= \frac{1}{n!|\mathcal{L}_n|}\sum_{L \in \mathcal{L}_n}\sum_{\pi \in S_n} \sum\limits_{\substack{M \in {\binom{[n]}{m}} \\ \pi(M) \subseteq [k]}} SS_m(L|_{\pi(M)}) \notag ,
    \end{split}
\end{equation}
where $S_n$ is the set of permutations of $[n]$. For every $L \in \mathcal{L}_n$, 

\begin{equation}
    \begin{split}
        \sum_{\pi \in S_n} \sum\limits_{\substack{M \in {\binom{[n]}{m}} \\ \pi(M) \subseteq [k]}}SS_m(\randLS|_{\pi(M)}) &= \binom{k}{m}(m!)\binom{n - m}{k - m}(k-m)!(n - k)!\sum_{M \in \binom{[n]}{m}}SS_m(L|_{M}) \\
        &= \left.n!\binom{n - m}{k - m}\sum_{M \in \binom{[n]}{m}}SS_m(L|_{M}) \middle/ \binom{n}{k}\right..\notag 
    \end{split}
\end{equation}
Combining the two equations above, we have that 

\begin{equation}
    \begin{split}
    \Expect{SS_{m}(\randLS|_{[k]})} &= \left.\binom{n-m}{k-m}\frac{1}{|\mathcal{L}_n|}\sum_{L \in \mathcal{L}_n}\sum_{M \in \binom{[n]}{m}}SS_m(LS|_{M}) \middle/ \binom{n}{k}\right. \\
    &= \left.\binom{n - m}{k - m}\Expect{SS_{m}(\randLS)} \middle/ \binom{n}{k}\right.\notag, 
    \end{split}
\end{equation}
as desired. 
\end{proof}

We will need the following lemma of Kwan and Sudakov~\cite[Proposition 5]{KS18}, which allows us to bound the expected number of Latin subsquares of a uniformly random Latin square in terms of the number of Latin subsquares of a uniformly random Latin rectangle.
For $L \in \setLR$, let $\mathcal{L}^*(L)$ be the set of order-$n$ Latin squares whose first $k$ rows coincide with $L$.
\begin{lemma}[Kwan and Sudakov~\cite{KS18}]\label{lem:extensionbound} For all $k \times n$ Latin rectangles, $L, L' \in \setLR$,
    $$\frac{|\mathcal{L}^*(L)|}{|\mathcal{L}^*(L')|} \leq e^{O(n \log^2n)}.$$
\end{lemma}

Using Lemmas~\ref{lem:prob-of-large-lss}--\ref{lem:extensionbound}, we can bound the expected number of large Latin subsquares in a uniformly random Latin square, as follows.

\begin{lemma}\label{lem:rbyrsubsquaremsymbolbound}
    Let $1/n \ll 1/c \ll 1$. 
    If $\randLS \in \mathcal{L}_n$ is a uniformly random order-$n$ Latin square and $m \geq c\sqrt{n \log n}$, then 
    $$\Expect{SS_m(\randLS)} \leq \exp(- n \log^2 n).$$
\end{lemma}
\begin{proof}

Let $k = \lceil n/3 \rceil$, and let $1 / c \ll 1 / C \ll 1$. By Lemma~\ref{lem:expectedvalue}, 
\begin{equation}\label{eq:probrxrssq}
    \Expect{SS_{m}(\randLS)} =  \Expect{SS_{m}(\randLS|_{[k]})}\left.\binom{n}{k}\middle/\binom{n-m}{k-m} \leq 2^n\Expect{SS_{m}(\randLS|_{[k]})}\right..
\end{equation}
By Lemma~\ref{lem:extensionbound}, for every partial $k \times n$ Latin rectangle $P$, 
\begin{equation*}
    \Prob{\randLS|_{[k]} \supseteq P} \leq \Prob{\mathbf{L'} \supseteq P}\exp(Cn \log^2 n),
\end{equation*}
where $\mathbf{L}'$ is a uniformly random $k \times n$ Latin rectangle. 
Therefore, by linearity of expectation,
\begin{equation} \label{eq:evssqrxr}
    \Expect{SS_{m}(\randLS|_{[k]})} \leq \Expect{SS_{m}(\mathbf{L}')}\exp(Cn \log^2 n).
\end{equation}
Combining \eqref{eq:probrxrssq} and \eqref{eq:evssqrxr}, if $m \geq n^{3/4}$, then by Lemma~\ref{lem:prob-of-large-lss},
\begin{equation*}
    \Expect{SS_m(\randLS)} \leq \exp(-n^{3/2}/10 + 2C n \log^2n) \leq \exp(-n\log^2 n),
\end{equation*}
as desired, and if $m \leq n^{3/4}$, then by Lemma~\ref{lem:prob-of-small-lss} with $c \sqrt{n \log n}$ playing the role of $m$,
    since   
    \begin{equation*}
        3^{c^2 n\log n}\left(c\sqrt{\frac{\log n}{n}}\right)^{c^2 n\log n - 3c\sqrt n \log n} \leq \exp\left(-\frac{1}{4}c^2 n \log^2 n\right),
    \end{equation*}
    we have
\begin{equation*}
    \Expect{SS_m(\randLS)} \leq \exp\left(-\frac{1}{4}c^2 n\log n + 2C n\log^2 n\right) \leq \exp(-n\log^2 n),
\end{equation*}
as desired.
\end{proof}

Theorem~\ref{thm:nolss} now follows easily from Lemma~\ref{lem:rbyrsubsquaremsymbolbound} and Markov's Inequality.

\begin{proof}[Proof of Theorem~\ref{thm:nolss}]
Let $\randLS \in \mathcal{L}_n$ be a uniformly random order-$n$ Latin square.
By Lemma~\ref{lem:rbyrsubsquaremsymbolbound} and linearity of expectation, the expected number of order-$m$ Latin subsquares for $m \geq c\sqrt{n \log n}$ is at most $n\exp(-n\log^2 n)$.   By Markov's Inequality, the probability $\randLS$ contains an order-$m$ Latin subsquare for $m \geq c\sqrt{n \log n}$ is also at most $n\exp(-n\log^2 n)$, as desired.
\end{proof} 

\bibliographystyle{amsabbrv}
\bibliography{ref}

\appendix

\section{Proof of Lemma~\ref{lem:granetjoos2.5}}\label{sec:appendix}
Our proof closely follows the proof of Granet and Joos \cite{GJ23}. First, we need the following lemmas from their paper. 

\begin{lemma}[Kühn and Osthus \cite{KO13}, Lemma 5.2]\label{lem:gjlem2.1}
    Let $0 \leq 1/n \ll \eps \ll \nu \leq \tau \ll \delta \leq 1$, and let $1/n \leq f \leq \eps$. If $D$ is a $(\delta, f)$-almost regular robust ($\nu, \tau)$-outexpander on $n$ vertices, then $D$ contains a spanning $(1 - f^{1/2})\delta n$-regular subdigraph which is a robust $(\nu /2, \tau)$-outexpander. 
\end{lemma}

\begin{lemma}[Joos and Kühn \cite{JK22}, Lemma 3.2]\label{lem:jklem3.2}
    Let $(X_t)_{t \geq 0}$ be a Markov chain with state space $[n]$, transition matrix $P$ with $P(i, j) \neq 0$ for all $i, j \in [n]$, and (unique) stationary distribution given by $(\sigma_i)_{i \in [n]}$ with $\sigma_i \neq 0$ for all $i \in [n]$. Let $\alpha:= \textnormal{min}_{i, j, k \in [n]}\frac{P(i,j)}{\sigma_k}$ and $\beta := \textnormal{max}_{i, j, k, \in [n]}\frac{P(i,j)}{\sigma_k}$. If $t \geq 2 + 2\alpha^{-1}\log \beta$ and $i \in [n]$, then
    $$\Prob{X_t = i} = \left(1 \pm \left(1 - \frac{\alpha}{2}\right)^t\right)\sigma_i.$$
\end{lemma}

\begin{lemma}[Granet and Joos \cite{GJ23}, Lemma 2.4]\label{lem:gjlem2.4}
    Let $0 < 1/n \ll \nu \leq \tau \ll \delta \leq 1$ and $\nu^{-1} + 1 \leq k \leq 2\nu^{-1}$ and $c \in \mathbb{N}$. If $D$ is a $\delta n$-regular robust $(\nu, \tau)$-outexpander on $[n]$, then letting $(X_t)_{t \geq 0}$ be the Markov chain corresponding to a random simple walk on $D$ with transition matrix $P$, the Markov Chain $(Y_t)_{t \geq 0} := (X_{kt + C})_{t \geq 0}$ has transition matrix $P^k$ with the uniform distribution as unique stationary distribution $(\sigma_i)_{i \in [n]}$. Moreover, 
    $$\nu^{k-1}\delta^{-k} \leq \frac{P^k(i, j)}{\sigma_\ell} \leq \delta^{-1}$$
    for any $i, j, \ell \in [n].$
\end{lemma}

Now we can prove Lemma~\ref{lem:granetjoos2.5}, which we restate for the reader's convenience.
\granetjoos*

\begin{proof}
        We only show that there are $(1 \pm \gamma)\delta^\ell n^{\ell - 1}$ $(u,v)$-paths for distinct $u, v \in V(G)$. The proof that each $u \in V(G)$ is in $(1 \pm \gamma)\delta^\ell n^{\ell - 1}$ length-$\ell$ cycles is nearly identical, so we omit it. By Lemma~\ref{lem:gjlem2.1}, $G$ contains a spanning $(1 - f^{1/2})\delta n$-regular subdigraph $D$ that is a robust $(\nu/2, \tau)$-outexpander. Note that $d^+_G(x) - d^+_D(x), d^-_G(x) - d^-_D(x) \leq 2\delta nf^{1/2}$ for every $x \in V(G)$. Therefore, the number of $(u, v)$-paths of length $\ell$ that contain at least one edge in $G\setminus D$ is at most $\ell \cdot 2 \delta nf^{1/2} \cdot (1 + f)^{\ell-2}\delta^{\ell-2}n^{\ell-2} = 2\ell f^{1/2}\delta^{\ell-1}n^{\ell-1}(1 + f)^{\ell - 2} \leq 2\eps^{1/2} \ell (1+\eps)^{\ell -2} \delta^{\ell - 1}n^{\ell - 1} \leq \gamma \delta^{\ell}n^{\ell - 1}/2$ because $f \leq \eps \ll 1/\ell \ll \delta$. 

    Note that $D$ contains $(1 - f^{1/2})^\ell\delta^\ell n^\ell$ walks of length $\ell$ starting at $u$. Let $k := \lceil 2\nu^{-1} \rceil$ + 1 and let $(X_t)_{t \geq 0}$ be the Markov chain corresponding to a simple random walk on $D$ starting at $u$, and define $(Y_t)_{t\geq 0} := (X_{\ell - k\lfloor \frac{\ell}{k} \rfloor+kt} )_{t \geq 0} $. Note that $\Prob{X_\ell = v} = \Prob{Y_{\lfloor\frac{\ell}{k}\rfloor}  = v}$, so the number of $(u, v)$-walks of length $\ell$ in $D$ is equal to
    $$\Prob{Y_{\lfloor \frac{\ell}{k}\rfloor} = v} \cdot (1 - f^{1/2})\delta^\ell n^\ell.$$
    By Lemmas \ref{lem:jklem3.2} and \ref{lem:gjlem2.4} (applied with $\frac{\nu}{2}$ playing the role of $\nu$ and $(1 - f^{1/2})\delta$ as $\delta$),
    $$\Prob{Y_{\lfloor \frac{\ell}{k}\rfloor} = v} = \left(1 \pm \left(1 - \frac{\nu^{k-1}\delta^{-k}}{2^{k}(1-f^{1/2})^{k}}\right)^{\lfloor \frac{\ell}{k}\rfloor}\right)n^{-1}.$$
    Since $f \ll 1/\ell \ll \gamma \ll \nu \ll \delta$,
    $$\left(1 - \frac{\nu^{k-1}\delta^{-k}}{2^k(1 - f^{1/2})^{k}}\right)^{\lfloor \frac{\ell}{k} \rfloor} \leq \frac{\gamma}{5}.$$

Combining the inequalities above, we have
    $$\Prob{Y_{\lfloor \frac{\ell}{k}\rfloor} = v} \cdot (1 - f^{1/2})\delta^\ell n^\ell = (1 \pm \gamma/4)\delta^\ell n^{\ell - 1}$$
$(u, v)$-walks of length $\ell$ in $D$. The number of these walks which are not paths is at most $\ell n^{\ell - 2} \leq \gamma \delta^{\ell}n^{\ell-1}/4$, which implies there are $(1 \pm \gamma/2)\delta^\ell n^{\ell - 1}$ $(u,v)$-paths of length $\ell$ in $D$. Hence, considering our bounds for the number of $(u,v)$-paths in $D$ and $(u,v)$-paths in $G$ containing at least one edge of $G \setminus D$, we see that $G$ contains $(1 \pm \gamma)\delta^\ell n^{\ell - 1}$ $(u,v)$-paths of length $\ell$, as desired.    
\end{proof}
\end{document}